\documentclass[a4paper,two-side,10pt]{article}
\usepackage[latin1]{inputenc}
\usepackage[T1]{fontenc}
\usepackage{amsmath}
\usepackage{amsthm}
\usepackage{latexsym}
\usepackage{amssymb}
\usepackage[all]{xy}
\usepackage{calc}
\usepackage{multicol}
\newtheorem{thm}{Theorem}[section]
\newtheorem{prop}{Proposition}[section]

\newtheorem{lem}{Lemma}[section]
\newtheorem{cor}{Corollary}[section]

\newtheorem{etape}{Step}%[section]
\newtheorem{rmq}{Remark}[section]

\newcommand{\R}{\mathbb{R}}
\numberwithin{equation}{section}
\newcommand{\N}{\mathbb{N}}

\newcounter{exercice}

\author{Jean-baptiste Castéras}
\title{ Equivariant mean field flow}
\begin{document}
\date{}
\maketitle
 \begin{center}
 \small{LMBA\\
 Universit\'e de Bretagne Occidentale\\
 6 av. Victor Le Gorgeu -- CS 93837\\
 29238 Brest Cedex\\
 France\\
 E-mail:  jean-baptiste.casteras@univ-brest.fr}
 \end{center}
\medskip
 \begin{abstract}
We consider a gradient flow associated to the mean field equation on $(M,g)$ a compact riemanniann surface without boundary. We prove that this flow exists for all time. Moreover, letting $G$ be a group of isometry acting on $(M,g)$, we obtain the convergence of the flow to a solution of the mean field equation under suitable hypothesis on the orbits of points of $M$ under the action of $G$.
 \end{abstract}
 \vspace{10pt}
 \begin{center} \textit{Key words} : Mean field equation, Geometric flow.
 \end{center}
 \begin{center} \textbf{AMS subjet classification} : 35B33, 35J20, 53C44, 58E20.
 \end{center}
\section{Introduction}
Let $(M,g)$ be a compact riemannian surface without boundary, we will study an evolution problem associated to the mean field equation :
\begin{equation}
\label{E:CM}
  \Delta u + \rho  \left( \dfrac{f  e^{u}}{\int_M f e^{u}dV} -\dfrac{1}{|M|}\right)=0, 
 \end{equation}
 where $\rho$ is a real parameter, $|M|$ stands for the volume of $M$ with respect to the metric $g$, $f\in C^\infty(M)$ is a given function supposed strictly positive and $\Delta$ is the Laplacian with respect to the metric $g$. The mean field equation appears in statistic mechanic from Onsager's vortex model for turbulent Euler flows. More precisely, in this setting, the solution $u$ of the mean field equation is the stream function in the infinite vortex limit (see \cite{MR1145596}). This equation is also linked to the study of condensate solutions of the abelian Chern-Simons-Higgs model (see for example \cite{MR1324400}, \cite{MR1955272}, \cite{MR1400816}, \cite{MR1838682}). Equation \eqref{E:CM} is also related to conformal geometry. When $(M,g)$ is the standard sphere and $\rho =8\pi$, the problem to find a solution to equation \eqref{E:CM} is called the Nirenberg Problem. The geometrical meaning of this problem is that, if $u$ is a solution of \eqref{E:CM}, the conformal metric $e^u g$ admits a gaussian curvature equal to $\dfrac{\rho f}{2}$.\newline
 
Equation (\ref{E:CM}) is the Euler-Lagrange equation of the nonlinear functional
 \begin{equation}
 \label{enerE:CM}
 I_\rho (u)=\frac{1}{2}\int_M |\nabla u|^2 dV +\frac{\rho}{|M|}\int_M udV -\rho \log \left( \int_M fe^u dV\right),\ u\in H^1(M). 
 \end{equation}
By using the well-known Moser-Trudinger's inequality (see inequality (\ref{equationmanquante2isocha})), one can easily obtain the existence of solutions of \eqref{E:CM} for $\rho <8\pi$ by minimizing $I_\rho$. Existence of solutions becomes much harder when $\rho\geq 8\pi$. In fact, in this case, the functional $I_\rho$ is not coercive. The existence of solutions to equation \eqref{E:CM} has been intensively studied these last decades when $\rho \geq 8\pi$. Many partial existence results have been obtained according to the value of $\rho$ and to the topology of $M$ (see for example \cite{MR2001443}, \cite{MR1712560}, \cite{MR2483132}, \cite{MR1619043} in the references therein). Recently, Djadli \cite{MR2409366} proves the existence of solutions to \eqref{E:CM} for all riemanniann surfaces when $\rho\neq 8k\pi$, $k\in \N^\ast$, by studying the topology of sublevels $\left\{I_\rho \leq -C\right\}$ to achieve a min-max scheme (already introduced in Djadli-Malchioldi \cite{MR2456884}).\newline

In this paper, we consider the evolution problem associated to the mean field equation, that is the following equation
\begin{equation}
\label{isoE:flot}
\left\{\begin{array}{ll}
\dfrac{ \partial}{\partial t}e^{u} =\Delta u + \rho  \left( \dfrac{f  e^{u}}{\int_M f e^{u}dV} -\dfrac{1}{|M|}\right)  \\
 u(x,0)=u_{0}(x),
\end{array}
\right.
\end{equation}
where $u_0 \in C^{2+\alpha}(M)$, $\alpha\in (0,1)$, is the initial data. It is a gradient flow with respect to the following functional :
\begin{equation}
\label{isoE:fonc}
E_f(u)= \frac{1}{2}\int_M|\nabla u|^2dV+\frac{\rho}{| M |}\int_M u dV- \rho \ln\left(\int_M f e^{u}dV\right), \ u\in H^1(M).
\end{equation}
We first prove the global existence of the flow \eqref{isoE:flot}. We obtain the following result :
\begin{thm}
\label{isoexglob}
For all $ u_0 \in C^{2+\alpha}(M)$ $(0<\alpha<1)$, all $\rho \in \R$ and all function $f\in C^\infty (M)$ strictly positive, there exists a unique global solution $u \in C_{loc}^{2+\alpha,1+\frac{\alpha}{2}}\left(M\times [0, + \infty)\right)$ of \eqref{isoE:flot}.
\end{thm}

Next, we investigated the convergence of the flow when the initial data and the function $f$ are invariant under an isometry group acting on $(M,g)$. A lot of works has been done for prescribed curvature problems invariant under an isometry group, we refer to \cite{MR2492192}, \cite{MR1216009}, \cite{MR0339258}, \cite{MR2807418} and the references therein. Before given a more precise statement of our results, we introduce some notations. Let $G$ be an isometry group of $\left(M,g\right)$. For all $x\in M$, we define $O_{G}(x)$ as the orbit of $x$ under the action of $G$, i.e.
$$O_{G}(x)=\left\{y \in M\ :\ y\in \sigma(x) \ ,\ \forall \sigma \in G\right\}.$$
$\left|O_{G}(x)\right|$ will stand for the cardinal of $O_{G}(x)$. We say that a function $f:M\rightarrow \R$ is $G$-invariant if $f\left(\sigma\left(x\right)\right)=f\left(x\right)$ for all $x\in M$ and $\sigma \in G$. We define $C_G^\infty (M)$ (resp. $C_G^{2+\alpha}(M)$, $\alpha\in (0,1)$) as the space of functions $f\in C^\infty (M)$ (resp. $f\in C_G^{2+\alpha}(M)$) such that $f$ is $G$-invariant. We prove the convergence of the flow under suitable hypothesis on $G$ allowing us also to handle the critical case when $\rho=8k\pi$, $k\in \N^\ast$.
\begin{thm}
\label{isoconvergence}
Let $G$ be an isometry group acting on $(M,g)$ such that
$$\left|O_{G}(x)\right|> \frac{\rho}{8\pi},\ \forall x \in M,$$
and $f\in C_G^\infty (M)$ be a strictly positive function. Then, for all initial $ u_0\in C_G^{2+\alpha}(M)$, the global solution $u\in C_{loc}^{2+\alpha,1+\frac{\alpha}{2}}(M\times [0,+\infty))$ of (\ref{isoE:flot}) converges in $H^2(M)$ to a function $u_\infty \in C_G^\infty (M)$ solution of the mean field equation (\ref{E:CM}).
\end{thm}
Assuming that $f$ is a positive constant and $G=Isom(M,g)$, the group of all isometry of $(M,g)$, we obtain :
\begin{cor}
Suppose that for all $x\in M$, we have $|O_G (x)|=+\infty$ then, for all $\rho\in \R$, the solution of the flow (\ref{isoE:flot}) converges in $H^2(M)$ to a function $u_\infty\in C^\infty_G(M)$ solution of the mean field equation (\ref{E:CM}).
\end{cor}
\begin{rmq}
Taking  $M=\mathbb{S}^1\times \mathbb{S}^1$ endowed with the product metric and $G=Isom (M,g)$, we have, for all $x\in M$, $|O_G(x)|=+\infty$.
\end{rmq}
If $f$ isn't constant, we also have :
\begin{cor}
If $\rho< 16\pi$ and $f\in C^\infty (\mathbb{S}^2)$ is an even function then the flow (\ref{isoE:flot}) converges in $H^2(\mathbb{S}^2)$ to an even function $u_\infty \in C^\infty (\mathbb{S}^2)$ solution of the mean field equation.
\end{cor}
The plan of this paper is the following : in Section $2$, we will prove Theorem \ref{isoexglob}. In Section $3$, we will give an improved Moser-Trudinger inequality for $G$-invariant functions. In Section $4$, we establish Theorem \ref{isoconvergence} : first, using our improved Moser-Trudinger inequality, we obtain a uniform (in time) $H^1(M)$ bound for the solution $u(t)$ of \eqref{isoE:flot} where $u(t): M\rightarrow \R$ is defined by $u(t)(x)=u(x,t)$. Then, from the previous estimate, we will derive a uniform $H^2(M)$ bound. Theorem \ref{isoconvergence} will follow from this last estimate.
\vspace{10pt}

\section {Proof of Theorem \ref{isoexglob}.}
In this section, we prove the global existence of the flow \eqref{isoE:flot}. We begin by noticing that, since the flow \eqref{isoE:flot} is parabolic, standard methods provide short time existence and uniqueness for it. Thus, there exists $T_1>0$ such that $u\in C^{2+\alpha,1+\frac{\alpha}{2}}(M\times [0,T_1])$ is a solution of the flow. It is also easy to see, integrating \eqref{isoE:flot} on $M$, that, for all $t\in [0,T_1]$, we have
\begin{equation}
\label{volconst}
\int_M e^{u(t)}dV=\int_M e^{u_0}dV.
\end{equation}
We also notice that the functional $E_f(u(t))$ is decreasing with respect to $t$. Differentiating $E_f (u(t))$ with respect to $t$ and integrating by parts, one finds, for all $t\in [0,T_1]$,
\begin{equation}
\label{isodecen}
\dfrac{\partial}{\partial t}E_f(u(t))= -\int_M \left( \dfrac{\partial u(t)}{\partial t} \right)^2 e^{u(t)}dV \leqslant 0 .
\end{equation}
Therefore, if $0\leq t_0\leq t_1\leq T_1$, we have 
\begin{equation}
\label{decene}
 E_f(u(t_1))\leqslant E_f(u(t_0)).
 \end{equation}
To prove Theorem \ref{isoexglob}, we set $$T=\sup\left\{\overline{T}>0 :\ \exists \ u \in C_{loc}^{2+\alpha, 1+\frac{\alpha}{2}}(M\times [0,\overline{T}])\ solution\ of\  (\ref{isoE:flot})\right\},$$
and suppose that $T<+\infty$. From the definition of $T$, we have that $u\in C_{loc}^{2+\alpha, 1+\frac{\alpha}{2}}(M\times [0,T))$. We will show that there exists a constant $C_T>0$ depending on $T$, $M$, $f$, $\rho$ and $\left\|u_0\right\|_{C^{2+\alpha}(M)}$ such that
\begin{equation}
\label{28juin2012e3}
\left\|u\right\|_{C^{2+\alpha, 1+\frac{\alpha}{2}}(M\times [0,T))}\leq C_T.
\end{equation}
This estimate allows us to extend $u$ beyond $T$, contradicting the definition of $T$. 
In the following, $C$ will denote constants depending on $M$, $f$, $\rho$ and $\left\|u_0\right\|_{C^{2+\alpha}(M)}$ and $C_T$ the ones depending on $M$, $f$, $\rho$, $\left\|u_0\right\|_{C^{2+\alpha}(M)}$ and $T$. 
\newcounter {EqNo}
 \setcounter{EqNo}{0}
\newcommand{\Numeq}{\refstepcounter{EqNo}  \hfill ( \thesection.\theEqNo)\\ }
\newcommand{\refeq}[1]{(\thesection.\ref{#1})}
\noindent 
\begin{prop}
\label{etape4}
There exists a constant $C_T$ such that, for all $t\in [0,T)$, we have
 \begin{equation}
 \label{etape4e}
 \left\|u(t)\right\|_{H^1 (M)}\leq C_T.
 \end{equation}
 \end{prop}
\begin{proof}
First, we claim that, for all $t\in [0,T)$, we have
 \begin{equation}
 \label{fardoune.g}
\int_M u^2(t)dV \leqslant C_1\int_M|\nabla u(t)|^2 dV +C_2,
\end{equation}
where $C_1$, $C_2$ are two constants depending on $T,\ f,\ \rho,\ \left\|u_0\right\|_{C^{2+\alpha}(M)},\ M$ and $A$ (where $A$ is the set defined in the following). Fix $t\in [0,T)$ and set  $$M_\varepsilon = \left\{x \in M : e^{u(x,t)}<\varepsilon\right\},$$
\noindent where $\varepsilon > 0$ is a real number which will be determined later. We set $\displaystyle\int_M e^{u_0}dV =a$. 
Using H\"{o}lder's inequality and \eqref{volconst}, one has
\begin{eqnarray}
\label{ee21.1}
a= \int_{M}e^{u(t)}dV & =& \int_{M_{\varepsilon}}e^{u(t)}dV+\int_{M\backslash M_{\varepsilon}} e^{u(t)}dV \nonumber\\
& \leq & \varepsilon |M_{\varepsilon}|+ \left|M\backslash M_{\varepsilon}\right|^{\frac{1}{2}} \left(\int_{M}e^{2u(t)}dV\right)^{\frac{1}{2}}.
\end{eqnarray}
Now, differentiating $\displaystyle\int_{M}e^{2u(t)}dV$  with respect to $t$, we get 
\begin{eqnarray}
\label{fardoun 2eme etape3.2}
\frac{1}{2}\frac{\partial}{\partial t}\left(\int_{M}e^{2u(t)}dV\right) & =& \int_{M}\Delta u(t) e^{u(t)}dV-\dfrac{\rho}{|M|}  \int_{M}e^{u(t)}dV+\dfrac{\rho \int_{M}fe^{2u(t)}dV}{\int_{M}fe^{u(t)}dV}\nonumber \\
& \leq &- \int_{M}|\nabla u(t)|^{2}e^{u(t)}dV +C +\dfrac{\rho \displaystyle\max_{x\in M}f(x)}{a\displaystyle\min_{x\in M}f(x)}\int_{M}e^{2u(t)}dV \nonumber\\
& \leq & C\int_{M}e^{2u(t)}dV+C,\  \forall t\in [0,T).
\end{eqnarray}
This yields to
\begin{equation}
\label{ee21.2}
\int_{M}e^{2u(t)}dV\leq C_T\ ,\ \forall t\in [0,T).
\end{equation}
From (\ref{ee21.1}), (\ref{ee21.2}) and taking $\varepsilon=\frac{a}{2|M|}$, we deduce that
\begin{equation}
\label{6juin2012e40}
|M\backslash M_{\varepsilon}|\geq \left(\dfrac{a}{2C_T} \right)^2 >0.
\end{equation}
Using Poincaré's inequality, we have  
 \begin{equation}
 \label{fardoun2.87bis}
 \int_M u^{2}(t)dV  \leq \frac{1}{\lambda_1}\int_M |\nabla u(t)|^2dV+ \dfrac{1}{|M|}\left(\int_M u(t)dV\right)^2,
 \end{equation}
 where $\lambda_1$ is the first eigenvalue of the laplacian. We set $A=M\backslash M_\varepsilon$. From \eqref{6juin2012e40} and since, for all $x\in M_1$ and $0\leq t<T$, 
$u(x,t)\geq \ln \varepsilon= \ln \left(\dfrac{a}{2|M|}\right),$ we deduce that there exists a constant $C_T$ such that
\begin{equation}
\label{rmqe1}
\int_{A}u(t)dV \geq  C_T.
\end{equation}
On the other hand, using \eqref{volconst}, we have 
$$\int_{A}u(t) dV \leq \int_{A}e^{u(t)} dV \leq a.$$ 
We deduce from the previous inequality and (\ref{rmqe1}) that there exists a constant $C_T$ such that
\begin{equation}
\label{rmqe2}
\left|\int_{A}u(t)dV  \right|\leq C_T.
\end{equation}
Now, using (\ref{rmqe2}) and Young's inequality, we have  

\begin{eqnarray}
\label{eg2.11bis}
&&\dfrac{1}{|M|}\left(\displaystyle \int_M u(t)dV \right)^2 \nonumber \\
& =& \dfrac{1}{|M|}\left( \displaystyle\int_{A} u(t)dV+\displaystyle\int_{M \setminus A} u(t) dV \right)^2 \nonumber\\
 &\leq &  \dfrac{1}{|M|} \left( C_T+2C_T \left|\displaystyle\int_{M \setminus A} u(t) dV\right| +\left(\displaystyle\int_{M \setminus A} u(t) dV \right)^2\right)\nonumber\\
&\leq & C_T+ \frac{2C_T \varepsilon_1+1}{|M|}\left(\int_{M\backslash A}u(t)dV\right)^{2},   
\end{eqnarray}
 where $\varepsilon_1$ is a strictly positive constant wich will be determined later. Using Cauchy-Schwarz's inequality, we obtain 
 \begin{equation}
 \label{eg2.12bis}
 \left(\int_{M\backslash A}u(t)dV\right)^{2} \leq  |M \backslash A| \int_{M\backslash A}u^2(t)dV.
 \end{equation}
\noindent Combining (\ref{fardoun2.87bis}), (\ref{eg2.11bis}) and (\ref{eg2.12bis}), we find

 \begin{eqnarray}
 \int_M u^2(t) dV  &\leq & \frac{1}{\lambda_1}\int_M |\nabla u(t)|^2 dV+C_T \nonumber\\
 &+& \left(1-\frac{ |A|}{|M|}+ 2C_T\varepsilon_1 \frac{|M\backslash A|}{|M|}\right) \int_{M \backslash A}u^2(t) dV.\nonumber
\end{eqnarray}
 Choosing $\varepsilon_1$ such that $\alpha= \left(1- \frac{|A|}{|M|}+ 2C_T\varepsilon_1 \frac{|M\backslash A|}{|M|}\right)<1,$ we have
 $$\left(1-\alpha \right)\int_M u^2(t) dV \leq \frac{1}{\lambda_1}\int_M |\nabla u(t)|^2 dV +C_T.$$
This shows that inequality (\ref{fardoune.g}) holds. From \eqref{volconst} and \eqref{decene}, we have
\begin{equation}
\label{eg2.20}
C_0:=E_f (u_0) \geq E_f (u(t))\geq \frac{1}{2}\int_M |\nabla u(t)|^2 dV +\dfrac{\rho}{|M|} \int_M  u(t)dV - C. 
\end{equation}
Using Young's inequality, we obtain
$$\frac{1}{2}\int_M |\nabla u(t)|^2 dV \leq C_0 + C+ \dfrac{\rho}{\varepsilon}+\varepsilon \int_M u^2(t) dV,$$
where $\varepsilon$ is a positive constant which will be determined later.
\noindent By (\ref{fardoune.g}), we have
$$ \frac{1}{2}\int_M |\nabla u(t)|^2 dV \leq C' +\varepsilon C_1  \int_M |\nabla u(t)|^2 dV,$$
where $C'=C_0 +C+ \dfrac{\rho}{\varepsilon}+ C_2$ and $C_1$, $C_2$ are the constants of \eqref{fardoune.g}. We choose $\varepsilon$ such that $\dfrac{1}{2}- \varepsilon C_1>0.$
\noindent Consequently, for all $t\in [0,T)$, we derive that
\begin{equation}
\label{fardoun2emeversionetoile}
\int_M |\nabla u(t)|^2 dV \leq C_T.
\end{equation}
Using once more (\ref{fardoune.g}) and (\ref{fardoun2emeversionetoile}), we have $\displaystyle\int_M u^2(t) dV\leq C_T$. Finally, we conclude that there exists a constant $C_T>0$ such that
$$\left\|u(t)\right\|_{H^1(M)}\leq C_T,\ \forall \ t\in [0,T).$$
\end{proof}

\begin{prop}
\label{etape5}
For all $\rho \in \R$, there exists a constant $C_T$ such that, for all $t\in [0,T)$, we have
 \begin{equation*}
 \left\|u(t)\right\|_{H^2 (M)}\leq C_T.
 \end{equation*}
 \end{prop}
\begin{proof}
In view of Proposition \ref{etape4}, we just need to bound $\displaystyle\int_M \left(\Delta u(t)\right)^2dV$, for all $t\in [0,T)$. We begin by setting 
$$v(t)=\dfrac{\partial u(t)}{\partial t}e^{\frac{u(t)}{2}},$$
then equation (\ref{isoE:flot}) becomes
$$v(t) e^{\frac{u(t)}{2}}=\Delta u(t) -\dfrac{\rho}{|M|} + \frac{\rho f e^{u(t)}}{\int_{M}f e^{u(t)}dV}.$$
Differentiating $\displaystyle\int_M \left(\Delta u(t)\right)^2dV$ with respect to $t$ and integrating by parts on $M$, we have
\begin{eqnarray*}
&& \frac{1}{2}\dfrac{\partial }{\partial t}\int_{M}\left(\Delta u(t)\right)^{2}dV  \\
& = &\int_{M}\left(v(t)e^{\frac{u(t)}{2}}+\dfrac{\rho}{|M|} - \frac{\rho f e^{u(t)}}{\int_{M}f e^{u(t)}dV}\right) \Delta \left(v(t)e^{-\frac{u(t)}{2}}\right)dV \\
& = &-\int_{M}\left|\nabla v(t)\right|^{2}dV+\frac{1}{4}\int_{M}v^{2}(t)\left|\nabla u(t)\right|^{2}dV \\
& +&\frac{\rho}{\int_{M}fe^{u(t)}dV}\left(\int_{M}\nabla f \nabla v(t) e^{\frac{u(t)}{2}}dV-\frac{1}{2}\int_{M}\nabla f v(t)\nabla u(t) e^{\frac{u(t)}{2}}dV\right.\\
&+&\left.\int_{M}f\nabla u(t) \nabla v(t) e^{\frac{u(t)}{2}}dV- \frac{1}{2}\int_{M} f \left|\nabla u(t)\right|^{2}v(t) e^{\frac{u(t)}{2}}dV\right).
\end{eqnarray*}

\noindent Since $f\in C^\infty (M)$ and is strictly positive (in particular we have $\displaystyle \int_{M}fe^{u(t)}dV\geq C \underset{x\in M}{\min} f(x)),$ we obtain

\begin{eqnarray}
\label{isoee48}
&&\frac{1}{2}\dfrac{\partial }{\partial t}\int_{M}\left(\Delta u(t)\right)^{2}dV\\ & \leq &-\int_{M}\left|\nabla v(t)\right|^{2}dV+C \int_{M}v^{2}(t)\left|\nabla u(t)\right|^{2}dV \nonumber \\
&+& C \left(\int_{M}e^{\frac{u(t)}{2}}\left(\left|\nabla v(t)\right|+ \left|v(t)\right|\left|\nabla u(t)\right|+\left|\nabla u(t)\right|\left| \nabla v(t)\right|+\left|\nabla u(t)\right|^{2}\left|v(t)\right|\right)dV\right). \nonumber
\end{eqnarray}
Let's estimate the positive terms on the right side of (\ref{isoee48}). From H\"{o}lder's inequality, we have, recalling that $\displaystyle\int_{M}e^{u(t)}dV=\int_M e^{u_0}dV$, for all $t\in [0,T)$,
\begin{equation}
\label{isoee48.2}
\int_{M}\left| \nabla v(t)\right|e^{\frac{u(t)}{2}}dV\leq C \left\|v(t)\right\|^{\frac{1}{2}}_{H^{1}(M)}.
\end{equation}
Using Gagliardo-Nirenberg's inequality (see for example \cite{MR1961176})
$$\left\|h\right\|^2_{L^4(M)}\leq C\left\|h\right\|_{L^2(M)}\left\|h\right\|_{H^1(M)},\ \forall h\in H^1(M),$$
and Cauchy-Schwarz's inequality, we have
\begin{eqnarray*}
\int_{M} v^2(t) \left|\nabla u(t)\right|^{2}dV & \leq &\left(\int_{M}v^{4}(t)dV\right)^{\frac{1}{2}} \left(\int_{M}\left|\nabla u(t)\right|^{4}dV\right)^{\frac{1}{2}} \nonumber\\
& = &  \left\|v(t)\right\|^{2}_{L^{4}(M)} \left\|\nabla u(t)\right\|^{2}_{L^{4}(M)} \nonumber\\
 & \leq & C\left\|v(t)\right\|_{L^{2}(M)} \left\|v(t)\right\|_{H^{1}(M)} \left\|\nabla u(t)\right\|_{L^{2}(M)}\left\|\nabla u(t)\right\|_{H^{1}(M)}. 
\end{eqnarray*}
Since, by Proposition \ref{etape4}, $\left\|u(t)\right\|_{H^1(M)}\leq C_T$ for all $t\in [0,T)$, we get
\begin{equation}
\label{isoee48.1}
\int_{M} v^2 (t)\left|\nabla u(t)\right|^{2}dV  \leq C_T \left\|v(t)\right\|_{L^{2}(M)} \left\|v(t)\right\|_{H^{1}(M)} \left\|u(t)\right\|_{H^2(M)}
\end{equation}
Using Proposition \ref{etape4} and Moser-Trudinger inequality \eqref{equationmanquante2isocha}, we deduce that there exists a constant $C_T$ such that, for all $ t\in [0,T)$,  and $p\in \R$,
\begin{equation}
\label{ibep}
\int_M e^{pv(t)}dV\leq C_T.
\end{equation}
In the same way as to prove \eqref{isoee48.1} and using \eqref{ibep}, we obtain
\begin{eqnarray}
\label{isoee48.3}
&&\int_{M} \left|\nabla u(t)\right| \left| v(t)\right|e^{\frac{u(t)}{2}}dV \nonumber \\
& \leq &\left(\int_{M}\left|\nabla u(t)\right|^{2}dV\right)^{\frac{1}{2}} \left(\int_{M}v^{4}(t)dV\right)^{\frac{1}{4}}  \left(\int_{M}e^{2u(t)}dV\right)^{\frac{1}{4}}\nonumber\\
& \leq &C_T \left\|v(t)\right\|_{L^{4}(M)} \nonumber\\
& \leq &C_T \left\|v(t)\right\|_{L^{2}(M)}^{\frac{1}{2}} \left\|v(t)\right\|_{H^{1}(M)}^{\frac{1}{2}},
\end{eqnarray}

\begin{eqnarray}
\label{isoee48.4}
&&\int_{M} \left|\nabla u(t)\right| \left| \nabla v(t)\right|e^{\frac{u(t)}{2}}dV \nonumber \\
& \leq &\left(\int_{M}\left|\nabla v(t)\right|^{2}dV\right)^{\frac{1}{2}} \left(\int_{M}\left|\nabla u(t)\right|^{4}dV\right)^{\frac{1}{4}}  \left(\int_{M}e^{2u(t)}dV\right)^{\frac{1}{4}}\nonumber\\
& \leq &C_T \left\|v(t)\right\|_{H^{1}(M)}\left\|u(t)\right\|_{H^{2}(M)}^{\frac{1}{2}},
\end{eqnarray}

and

\begin{eqnarray}
\label{isoee48.5}
&&\int_{M} \left|\nabla u(t)\right|^{2} \left| v(t)\right|e^{\frac{u(t)}{2}}dV \nonumber \\
& \leq &C_T \left\|v(t)\right\|_{L^{4}(M)}\left\|\nabla u(t)\right\|_{L^{4}(M)}^2\nonumber\\
& \leq &C_T \left\|u(t)\right\|_{H^{2}(M)}\left\|v(t)\right\|_{L^{2}(M)}^{\frac{1}{2}}\left\|v(t)\right\|_{H^{1}(M)}^{\frac{1}{2}}.%\nonumber \\
\end{eqnarray}
Finally, inserting estimates \eqref{isoee48.2}, \eqref{isoee48.1}, \eqref{isoee48.3}, \eqref{isoee48.4}, \eqref{isoee48.5} into \eqref{isoee48}, it follows that
\begin{eqnarray*}
\frac{1}{2}\dfrac{\partial }{\partial t}\int_{M}\left(\Delta u(t)\right)^{2}dV & \leq &- \int_M |\nabla v(t)|^2dV  + C_T \left\|v(t)\right\|_{L^{2}(M)}\left\|v(t)\right\|_{H^{1}(M)}\left\|u (t)\right\|_{H^{2}(M)} \\
&+& C_T \left(\left\|v(t)\right\|_{H^{1}(M)}^{\frac{1}{2}}+ \left\|v(t)\right\|_{H^{1}(M)}^{\frac{1}{2}}\left\|v(t)\right\|_{L^{2}(M)}^{\frac{1}{2}}\right)\\
&+&C_T\left\|v(t)\right\|_{H^{1}(M)}\left\|u(t)\right\|_{H^{2}(M)}^{\frac{1}{2}}\\
&+&C_T \left\|u(t)\right\|_{H^{2}(M)}\left\|v(t)\right\|_{L^{2}(M)}^{\frac{1}{2}}\left\|v(t)\right\|_{H^{1}(M)}^{\frac{1}{2}}.
\end{eqnarray*}
Using Young's inequality on each positive terms, we obtain
\begin{eqnarray}
\label{6juin2012e20}
&&\dfrac{\partial }{\partial t}\left(\int_{M}\left(\Delta u(t)\right)^{2}dV+ 1\right)  \nonumber \\
&\leq & C_T \left(\int_{M}\left(\Delta u(t)\right)^{2}dV+ 1\right)\left(\left\|v(t)\right\|_{L^{2}(M)}+1 \right).
\end{eqnarray}
\noindent On the other hand, for all $t\in [0,T)$, one has, since $\left\|u(t)\right\|_{H^1(M)}\leq C_T$,
\begin{eqnarray}
\label{iso51bis}
\int_0^t\left\|v(s)\right\|_{L^{2}(M)}^{2}ds&=&\int_0^t\int_{M}v^{2}(s)dVds\nonumber\\
&=&\int_0^t\int_{M}\left(\dfrac{\partial u(s)}{\partial t}\right)^{2}e^{u(s)}dVds\nonumber \\
&=&-\int_0^t\dfrac{\partial}{\partial t} E_f(u(s))ds=E_f(u_0)-E_f(u(t))\nonumber \\
 &\leq &C_T.
\end{eqnarray}
Thus, integrating \eqref{6juin2012e20} with respect to $t$ and using (\ref{iso51bis}), it follows that 
$$\int_{M}\left(\Delta u(t)\right)^{2}dV\leq C_T,\ \forall t\in [0,T).$$
Therefore we conclude that
$$\left\|u(t)\right\|_{H^{2}(M)}\leq C_T, \forall  t\in \left[0,T\right).$$
\end{proof}
\vspace{12pt}
\vspace{12pt}
\noindent\textit{Proof of Theorem \ref{isoexglob}.} We recall that to prove Theorem \ref{isoexglob}, it is sufficient to prove \eqref{28juin2012e3}, i.e. there exists a constant $C_T>0$ such that
\begin{equation*}
%\label{28juin2012e3}
\left\|u\right\|_{C^{2+\alpha, 1+\frac{\alpha}{2}}(M\times [0,T))}\leq C_T.
\end{equation*}
First, we claim that for all $\alpha\in (0,1)$, there exists a constant $C_T$ such that
\begin{equation}
\label{6juin2012e11}
|u(x_1,t_1)-u(x_2,t_2)|\leq C_T(|t_1-t_2|^{\frac{\alpha}{2}}+|x_1-x_2|^{\alpha}),
\end{equation}
for all $x_1,x_2 \in M$ and all $t_1,t_2 \in [0,T)$, where $|x_1-x_2|$ stands for the geodesic distance from $x_1$ to $x_2$ with respect to the metric $g$. From Proposition \ref{etape5} and Sobolev's embedding Theorem, we get, for $\alpha \in (0,1)$, $t\in [0,T)$ that there exists a constant $C_T$ such that $\left\|u(t)\right\|_{C^\alpha (M)}\leq C_T$, i.e. for all $x,y\in M$,
\begin{equation}
\label{bren4}
|u(x,t)-u(y,t)|\leq C_T |x-y|^\alpha .
\end{equation}
If $t_2-t_1 \geq 1$, using \eqref{bren4}, it is easy to see that \eqref{6juin2012e11} holds. Therefore, from now on, we assume that $0<t_2-t_1<1$. On the other hand, since $u(t)$ is a solution of \eqref{isoE:flot} and $\left\|u(t)\right\|_{C^\alpha (M)}\leq C_T$, one has, $\forall t\in [0,T)$,
$$\left|\dfrac{\partial u(t)}{\partial t}\right|^2\leq C_T |\Delta u(t)|^2 +C_T.$$
Integrating the previous estimate on $M$, we obtain, for all $t \in [0,T)$,
\begin{equation}
\label{bren1}
\int_M \left|\dfrac{\partial u(t)}{\partial t}\right|^2dV \leq C_T \left\|u(t)\right\|^2_{H^2(M)}+C_T\leq C_T.
\end{equation}
Now, we write 
\begin{eqnarray}
\label{bren3}
|u(x,t_1)&-&u(x,t_2)|= \frac{1}{|B_{\sqrt{t_2-t_1}}(x)|}\int_{B_{\sqrt{t_2-t_1}}(x)} |u(x,t_1)-u(x,t_2)|dV(y)\nonumber \\
&\leq & \frac{C}{ t_2-t_1} \int_{B_{\sqrt{t_2-t_1}}(x)} |u(x,t_1)-u(y,t_1)|dV(y)\nonumber\\
& +&\frac{C}{ t_2-t_1}\int_{B_{\sqrt{t_2-t_1}}(x)} |u(y,t_1)-u(y,t_2)|dV(y) \nonumber\\
&+&\frac{C}{ t_2-t_1}\int_{B_{\sqrt{t_2-t_1}}(x)} |u(y,t_2)-u(x,t_2)|dV(y),
\end{eqnarray}
where $B_{\sqrt{t_2-t_1}}(x)$ stands for the geodesic ball of radius $\sqrt{t_2-t_1}$ centered in $x$. \noindent Let's consider the first term on the right of (\ref{bren3}). Using (\ref{bren4}), we obtain
\begin{eqnarray}
\label{bren5}
&&\frac{C}{t_2-t_1} \int_{B_{\sqrt{t_2-t_1}}(x)} |u(x,t_1)-u(y,t_1)|dV(y)\nonumber \\ &\leq & \frac{C_T}{ (t_2-t_1)} \int_{B_{\sqrt{t_2-t_1}}(x)} |x-y|^{\alpha}dV(y)\nonumber\\
&\leq & C_T(t_2-t_1)^{\frac{\alpha}{2}}.
\end{eqnarray}
In the same way, we have
\begin{equation}
\label{bren6}
\frac{C}{ t_2-t_1} \int_{B_{\sqrt{t_2-t_1}}(x)} |u(x,t_2)-u(y,t_2)|dV(y) \leq C_T (t_2-t_1)^{\frac{\alpha}{2}}.
\end{equation}
We have, using H\"{o}lder's inequality and (\ref{bren1}),
\begin{eqnarray}
\label{bren2}
&&\frac{C}{t_2-t_1}\int_{B_{\sqrt{t_2-t_1}}(x)} |u(y,t_1)-u(y,t_2)|dV(y)\nonumber \\ &\leq & C \sup_{t_1\leq\tau \leq t_2 }\int_{B_{\sqrt{t_2-t_1}}(x)} \left|\dfrac{\partial u}{\partial s}  \right| (y,\tau)dV(y)\nonumber \\
& \leq & C \sqrt{t_2-t_1} \sup_{t_1\leq\tau \leq t_2 }\left(\int_{B_{\sqrt{t_2-t_1}}(x)} \left|\dfrac{\partial u}{\partial s}  \right|^2 (y,\tau)dV(y)\right)^{\frac{1}{2}}\nonumber \\
& \leq & C_T \sqrt{t_2-t_1}.
\end{eqnarray}
Putting (\ref{bren5}), (\ref{bren6}), (\ref{bren2}) in (\ref{bren3}) and noticing that for all $0<t_2-t_1<1$, we have $\sqrt{t_2-t_1}\leq (t_2-t_1)^{\frac{\alpha}{2}}$, we find
\begin{equation}
\label{bren7}
|u(x,t_1)-u(x,t_2)|\leq  C_T (t_2-t_1)^{\frac{\alpha}{2}} .
\end{equation}
The inegalities (\ref{bren4}) and (\ref{bren7}) imply that 
\begin{equation*}
|u(x_1,t_1)-u(x_2,t_2)|\leq C_T(|t_1-t_2|^{\frac{\alpha}{2}}+|x_1-x_2|^{\alpha}),
\end{equation*}
for all $x_1,x_2 \in M$ and all $t_1,t_2 \in [0,T)$, $0<t_2-t_1 <1$. This establishes \eqref{6juin2012e11}.
In view of \eqref{6juin2012e11}, we may apply the standard regularity theory for parabolic equations (see for example \cite{MR0181836}) to derive that there exists a constant $C_T$ depending on $T$ such that
$$\left\|u\right\|_{C^{2+\alpha,1+\frac{\alpha}{2}}(M\times [0,T))}\leq C_T,\ \alpha\in (0,1).$$
%Therefore, we obtain a contradiction with the definition of $T$. That is to say, we have $T=+\infty$ so we obtain that $u\in C_{loc}^{2+\alpha,1+\frac{\alpha}{2}}(M\times [0,T))$. 
This completes the proof of Theorem \ref{isoexglob}. 
 
\section{Improved Moser-Trudinger's inequality.}
We begin by recalling the Moser-Trudinger's inequality (see \cite{MR0301504}, \cite{MR0216286}): there exists a constant $C$ depending on $(M,g)$ such that, for all $u\in H^1 (M)$,
\begin{equation}
\label{realmosertrud}
\int_M e^{\frac{4\pi (u -\bar{u})^2}{\int_M |\nabla u|^2 dV}}dV\leq C,
\end{equation}
where $\bar{u}$ stands for the mean value of $u$ on $M$ i.e. $\bar{u}=\dfrac{\int_M u dV}{|M|}$. As a consequence of \eqref{realmosertrud}, we obtain the following inequality : there exists a constant $C$ depending on $(M,g)$ such that, for all $u\in H^1 (M)$, 
\begin{equation}
\label{equationmanquante2isocha}
\log \left(\int_{M}e^{u-\bar{u}}dV\right)\leq \frac{1}{16\pi}\int_{M}\left|\nabla u\right|^{2}dV+C.
\end{equation}

\noindent The next lemma shows that the Moser-Trudinger's inequality \eqref{equationmanquante2isocha} can be improved for functions which "concentrate" in some points.
\begin{lem}[\cite{MR1129348}]
\label{iil1.1}
Let $\delta_{0}$, $\gamma_{0}$ be some positive real numbers, $l$ an integer, $\Omega_{1},\ldots,\Omega_{l}$ subsets of $M$ such that $dist\left(\Omega_{i},\Omega_{j}\right)\geq \delta_{0}$ for $i\neq j$. Then, for all $\tilde{\varepsilon}>0$, there exists a constant $C$ depending on $l,\tilde{\varepsilon},\delta_{0}$ and $\gamma_{0}$ such that
$$\log \left(\int_{M}e^{u-\bar{u}}dV\right)\leq \left(\frac{1}{16l\pi}+\tilde{\varepsilon}\right) \int_{M}\left|\nabla u\right|^{2}dV+C,$$
for $u \in H^{1}\left(M\right)$ such that, for all $i\in \left\{1,\ldots,l\right\}$,
$$\int_{\Omega_{i}}e^{u}dV\geq \gamma_{0}\int_{M}e^{u}dV.$$
\end{lem}
We define $H_{G}^{1}(M)$ the space of functions $u\in H^1(M)$ such that $u$ is $G$-invariant. By considering functions $u\in H_{G}^{1}(M)$ and assuming hypothesis on the cardinal of orbits of points of $M$ under the action of $G$, we find that there exist subsets $\left\{\Omega_{i}\right\}_{1\leq i\leq l}$ of $M$ such as described previously. More precisely, we have

\begin{prop}
\label{iil1.2}
Let $k\in \N$, $k\geq 2$. Suppose that $\displaystyle \min_{x\in M}\left|O_{G}\left(x\right)\right|\geq k$. Then, for all $\varepsilon>0$, there exists a constant $C$ positive depending on $M$ and $\varepsilon$ such that, for all $ u\in H_{G}^{1}(M)$,
\begin{equation}
\label{improvedmoser}
\log \left(\int_{M}e^{u-\bar{u}}dV\right)\leq  \left(\frac{1}{16k\pi}+\varepsilon\right) \int_{M}\left|\nabla u\right|^{2}dV+C.
\end{equation}
\end{prop}

\begin{proof}
The proof will be divided into two steps.
\begin{etape}
\label{ietape1ineame}
Let $u\in H^{1}_{G}(M)$ then there exists $x_{1}\in M$ (depending on $u$) such that, $\forall 0<r<i(M)$,
$$\int_{B_{r}\left(x_{i}\right)}e^{u}dV\geq C r^{2}\int_{M}e^{u}dV,$$
where $i(M)$ stands for the injectivity radius of $(M,g)$, $B_{r}\left(x_{i}\right)$ stands for the geodesic ball of radius $r$ centered in $x_i$ and $C$ is a positive constant depending on $(M,g)$.
\end{etape}
\noindent\textit{Proof of Step \ref{ietape1ineame}.} Let $0<r<i(M)$. Since $M$ is compact, there exists a finite number of points $x_{1},\ldots,x_{m}\in M$ such that
$$M=\displaystyle\bigcup_{i=1}^{m} B_{r}(x_{i})\ ;\ B_{\frac{r}{2}}(x_{i})\cap B_{\frac{r}{2}}(x_{i})=\varnothing, \ \forall i\neq j.$$
We can assume, up to relabelling the $x_i$'s, $i\in \left\{1,\ldots, m\right\}$, that $$\int_{B_{r}(x_{1})}e^{u}dV=\displaystyle \max_{i\in \left\{1,\ldots,m\right\}}\int_{B_{r}(x_{i})}e^{u}dV.$$ Therefore, we have
\begin{equation}
\label{iie1.1}
\int_{M}e^{u}dV \leq \displaystyle\sum_{i=1}^{m} \int_{B_{r}(x_{i})}e^{u}dV  \leq m \int_{B_{r}(x_{1})}e^{u}dV.
\end{equation}
On the other hand, there exists $C$, a positive constant depending only on $(M,g)$ such that, for all $ 1\leq i\leq m$,
$$\left|B_{\frac{r}{2}}(x_{i})\right|\geq C r^{2},$$
thus
\begin{equation}
\label{iie1.2}
 \left|M\right|\geq \sum_{i=1}^{m} \left|B_{\frac{r}{2}}(x_{i})\right|\geq m C r^{2}.
 \end{equation}
 Combining (\ref{iie1.1}) and (\ref{iie1.2}), we obtain
 $$\int_{B_{r}(x_{1})}e^{u}dV\geq \frac{Cr^{2}}{\left|M\right|} \int_{M}e^{u}dV.$$
% \end{proof}

\begin{etape}
\label{ietape2ineame}
Let $G$ be an isometry group such that, for all $x\in M$, $|O_G(x)|\geq k$. There exists $\delta$ depending on $M$, $G$ and $k$ such that, for all $x\in M$, there exist $k$ points $x=x_{1},\ldots,x_{k} \in O_{G}(x)$ such that
  $$|x_{i}-x_{j}|\geq \delta \ si\ i\neq j,$$
 where $|x_{i}-x_{j}|$ stands for the geodesic distance between $x_{i}$ and $x_{j}$ with respect to the metric $g$.
\end{etape}  

\noindent\textit{Proof of Step \ref{ietape2ineame}.}
 We proceed by induction on $k\geq 2$. Suppose that $k=2$. By contradiction, we assume that $\forall n \in \N^{\ast},\ \exists \ x_{n}\in M$ such that $$|x_{n}-\sigma (x_{n})|<\frac{1}{n},\ \forall \ \sigma\in G.$$ Since $M$ is compact, there exists $x\in M$ such that $|x_{n}-x|\underset{n\rightarrow +\infty}{\longrightarrow} 0$. Therefore, we find $|x-\sigma (x)|=0$ for all $\sigma\in G$. This implies that $\left|O_{G}(x)\right|=1$. This provides us a contradiction.\newline
Now suppose that the induction hypothesis holds for groups $G$ such that  $\left|O_{G}(x)\right|\geq k$ for all $x\in M$, i.e.
 \begin{eqnarray}
 \label{iie1.3}
 &&\exists \delta_{k}>0\ s.t.\ \forall x\in M,\ \exists k\ points\ x=x_{1},\ldots,x_{k} \in O_{G}(x)\ \nonumber\\ &&  \qquad such\ that\  |x_{i}-x_{j}|\geq \delta_{k} \qquad,\qquad \forall i\neq j\in \{1,\ldots ,k\}.
\end{eqnarray}
Let's consider $G$ a group such that $\left|O_{G}(x)\right|\geq k+1$ for all $x\in M$. Proceeding by contradiction, from (\ref{iie1.3}), we see that $\forall n\in \N^{\ast}$, $\exists \ x^{1}_{n}\in M$ and $x^{2}_{n},\ldots,x^{k}_{n}\in O_{G}(x^{1}_{n})$ such that
$$|x^{i}_{n}- x^{j}_{n}|\geq \delta_{k},\ \forall i\neq j\in \{1,\ldots ,k\},$$
and, $\forall \sigma\in G$,
\begin{equation}
\label{iie1.4}
\inf_{j \in \left\{1,\ldots,k\right\}}|\sigma (x^{1}_{n})-x^{j}_{n}|<\frac{1}{n}. 
\end{equation}
Since $M$ is compact, we can assume that there exist $ x^{1},\ldots,x^{k}\in M$ such that
$$ \left|x^{j}_{n}- x^{j}\right|\underset{n\rightarrow +\infty}{\longrightarrow} 0,\ \forall j\in \{1,\ldots,k\}.$$
Letting $n$ tend to $+\infty$ in (\ref{iie1.4}), we deduce that, $\forall \sigma\in G$,
$$ \inf_{j \in \left\{1,\ldots,k\right\}}|\sigma (x^{1})-x^{j}|=0 .$$
This implies that
$$O_{G}(x^{1})\subseteq \left\{x^{1},\ldots,x^{k}\right\}.$$
Therefore we get a contradiction with the fact that $\left|O_{G}(x) \right|\geq k+1$ for all $x\in M$.
%\end{proof}
\vspace{12pt}
\vspace{12pt}

\noindent\textit{Proof of Lemma \ref{iil1.2}.} %We are now in position to prove Lemma \ref{iil1.2}. 
Let $u\in H_{G}^{1}(M)$. From Step \ref{ietape1ineame}, there exists $x_{1}\in M$ such that, for all $0<r<i(M)$,
$$\int_{B_{r}\left(x_{i}\right)}e^{u}dV\geq C r^{2}\int_{M}e^{u}dV.$$
By Step \ref{ietape2ineame}, there exist a constant $\delta>0$ depending on $M$, $k$, $G$ and points $x_{2},\ldots,x_{k} \in O_{G}(x_{1})$ such that $|x_{i}-x_{j}|\geq \delta$, for all $i\neq j\in \{1,\ldots ,k\}$. We can suppose that $\delta < i(M)$. Setting $\Omega_{i}=B_{\frac{\delta}{4}}(x_{i})$, we have
\begin{equation}
\label{28juin2012e1}
dist(\Omega_{i},\Omega_{j})\geq \frac{\delta}{2}.
\end{equation}
Since $x_{j}\in O_{G}(x_{1})$ for all $j\in \{2,\ldots,k\}$, there exists $\sigma_{j}\in G$ such that $x_{j}=\sigma_{j}(x_{1})$. Using the $G$-invariance of $u$, we have, for all $j\in \{2,\ldots,k\}$,
\begin{equation}
\label{28juin2012e2}
\int_{\Omega_{j}}e^{u}dV=\int_{\Omega_{1}}e^{u}dV\geq C \delta^{2}\int_{M}e^{u}dV.
\end{equation}
From \eqref{28juin2012e1} and \eqref{28juin2012e2}, the hypothesis of Lemma \ref{iil1.1} are satisfied. The proof follows immediately.
\end{proof}

\section{Proof of Theorem \ref{isoconvergence}.}
This section is devoted to the convergence of the flow when its initial data $u_0\in C^{2+\alpha}(M)$, $\alpha\in (0,1)$, and the function $f\in C^\infty (M)$ are invariant under the action of an isometry group $G$ acting on $(M,g)$. Let $u: M\times [0,+\infty )\rightarrow \R$ be the unique global solution of \eqref{isoE:flot}. We begin by noticing, since $u_0\in C_G^{2+\alpha}(M)$, $\alpha\in (0,1)$, and $f\in C_G^\infty$, that $u(t)$ is $G$-invariant for all $t\geq 0$. First, we prove that $u(t)$, $t\geq 0$, is uniformly (in time) bounded in $H^1(M)$. 
In the following, $k$ will stand for $\displaystyle\min_{x\in M}|O_G(x)|$ and $C$ will denote constants depending on $M$, $f$, $\rho$ and $\left\|u_0\right\|_{C^{2+\alpha}(M)}$.
\begin{prop}
\label{propisoh1}
Let $\rho < 8k\pi$. Then, there exists a constant $C$ such that
\begin{equation}
\label{isoborneh1}
\left\|u(t)\right\|_{H^{1}(M)}\leq C,\ \forall t\geq 0.
\end{equation}
\end{prop}
\begin{proof}
From \eqref{decene}, we have
\begin{equation}
\label{iie2.1}
E_{f}\left(u\left(t\right)\right)\leq E_{f}\left(u_{0}\right):=C_0,\ \forall t\geq 0 .
\end{equation}
To prove (\ref{isoborneh1}), we will consider two cases $\rho<0$ and $0\leq \rho <8k\pi$. In a first time, let's consider the case $\rho<0$. Since $f\in C^\infty(M)$ and is a strictly positive function, we have
\begin{eqnarray*}
E_{f}(u(t)) & =& \frac{1}{2}\int_{M}\left|\nabla u(t)\right|^{2}dV+ \dfrac{\rho}{|M|}\int_{M}u(t)dV-\rho \log\left(\int_{M}fe^{u(t)}dV\right)\nonumber \\
& \geq &\frac{1}{2}\int_{M}\left|\nabla u(t)\right|^{2}dV-\rho \log\left(\int_{M}e^{u(t)-\bar{u}(t)}dV\right)-C.\nonumber\\
\end{eqnarray*}
Using Jensen's inequality, 
it follows that
$$E_f(u(t))\geq \frac{1}{2}\int_{M}\left|\nabla u(t)\right|^{2}dV-C.$$
From (\ref{iie2.1}), we find
\begin{equation}
\label{iie2.2}
\int_M |\nabla u(t)|^2 dV\leq C,\ \forall t\geq 0.
\end{equation}
Now let's consider the second case $0\leq \rho< 8k\pi$. Since, as we already notice, $u(t)$ is $G-$invariant, for all $t\geq 0$, we can use the improved Moser-Trudinger's inequality \eqref{improvedmoser} of Lemma \ref{iil1.2}. This gives us

\begin{eqnarray}
\label{ie2.2'}
E_{f}(u(t)) & \geq &\frac{1}{2}\int_{M}\left|\nabla u(t)\right|^{2}dV-\rho \log\left(\int_{M}e^{u(t)-\bar{u}(t)}dV\right)-C\nonumber\\
& \geq &\left(\frac{1}{2}-\frac{\rho}{16k\pi}-\varepsilon\right)\int_{M}\left|\nabla u(t)\right|^{2}dV-C.
\end{eqnarray}

\noindent Since $0\leq \rho < 8k\pi$, we obtain, taking $\varepsilon=\dfrac{8k\pi-\rho}{32k\pi}$, using (\ref{iie2.1}) and  (\ref{ie2.2'}),
\begin{equation}
\label{equationmanquante1juillet}
\int_{M}\left|\nabla u(t)\right|^{2}dV\leq C,\ \forall t\geq 0.
\end{equation}
From \eqref{iie2.2} and \eqref{equationmanquante1juillet}, we deduce that, for all $\rho <8k\pi$,
\begin{equation}
\label{equationmanquante1}
\int_{M}\left|\nabla u(t)\right|^{2}dV\leq C,\ \forall t\geq 0.
\end{equation}
Now, using Poincaré's inequality, we get
\begin{equation}
\label{uborh1}
 \left\|u(t) -\bar{u}(t) \right\|_{H^{1}(M)}\leq C.
\end{equation}
From the improved Moser-Trudinger's inequality \eqref{improvedmoser} and (\ref{equationmanquante1}), we have
$$\int_{M}e^{u(t)-\bar{u}(t)}dV\leq C.$$
Since, $\displaystyle\int_{M}e^{u(t)}dV=\int_M e^{u_0}dV $, for all $t\geq 0$, we find
\begin{equation}
\label{ubarbor} \bar{u}(t)\geq -C,\ \forall t\geq 0.\end{equation}
On the other hand, using Jensen's inequality, we have
\begin{equation}
\label{ubarbor1}
 \bar{u}(t)\leq -C,\ \forall t\geq 0.\end{equation}
Finally, from (\ref{uborh1}), (\ref{ubarbor}) and (\ref{ubarbor1}), we conclude that 
\begin{equation}
\label{iibh1}
\left\|u(t)\right\|_{H^{1}(M)}\leq C,\ \forall  t \geq 0.
\end{equation}
From \eqref{improvedmoser} and \eqref{iibh1}, we deduce that there exists a constant $C$ (not depending on time) such that, for all $t\geq 0$ and $p\in \R$,
\begin{equation}
\label{ibep1}
\int_{M}e^{pu(t)}dV\leq C.
\end{equation} 
We also notice that integrating
$\dfrac{\partial}{\partial t}E_{f}(u(t))=-\displaystyle\int_{M} \left|\dfrac{\partial}{\partial t}u(t)\right|^{2}e^{u(t)}dV$
with respect to $t$ and since $\left\|u(t)\right\|_{H^1(M)}\leq C$, we find, for all $T_1 \geq 0$,
\begin{equation}
\label{equationmanquante3}
\int_{0}^{T_1}\int_{M}\left|\dfrac{\partial}{\partial t}u(t)\right|^{2}e^{u(t)}dVdt= E_{f}(u_{0})-E_f(u(T_1))\leq C.
\end{equation}
\end{proof}
\vspace{12pt}
\vspace{12pt}

\noindent\textit{Proof of Theorem \ref{isoconvergence}.} We follow closely Brendle \cite{MR1999924} arguments. We set $U(t)=\dfrac{\partial}{\partial t}u(t)$ and $y(t)=\displaystyle\int_{M}U^{2}(t)e^{u(t)}dV.$ We claim that
$$y(t)\underset{t\rightarrow +\infty}{\longrightarrow} 0.$$
Let $\varepsilon$ be some real positive number. In view of \eqref{equationmanquante3}, there exists $t_{0}> 0$ such that $y(t_{0})\leq \varepsilon$. We want to show that 
\begin{equation}
\label{6juin2012e2}
y(t)\leq 3\varepsilon , \ \forall t\geq t_{0}.
\end{equation}
Suppose, by contradiction, that \eqref{6juin2012e2} doesn't hold. We set
$$t_{1}=\inf \left\{\ t\geq t_{0}\ :\ y(t)\geq 3\varepsilon \right\}<+\infty.$$
By definition, we have
$$y(t)\leq 3\varepsilon, \ \forall t_{0}\leq t\leq t_{1}.$$
Since $\dfrac{\partial u(t)}{\partial t}=e^{-u(t)}\left(\Delta u(t) -\dfrac{\rho}{\left|M\right|}\right)+\dfrac{\rho f}{\int_{M}fe^{u(t)}dV}$, we find, using Young's inequality, for all $t_{0}\leq t\leq t_{1}$,
\begin{eqnarray}
\label{iie3.1}
&&\int_{M}e^{-u(t)}\left(\Delta u(t) -\frac{\rho}{\left| M\right|}\right)^{2}dV\nonumber \\
&=& y(t)-2 \dfrac{\rho}{\int_M fe^u dV}\int_M f\dfrac{\partial u}{\partial t}(t) e^{u(t)}dV+\int_{M}\left(\frac{\rho f}{\int_{M}fe^{u(t)}dV}\right)^{2}e^{u(t)}dV\nonumber \\
& \leq & C y(t)+C\int_M e^{u(t)}dV \leq  C_1,
\end{eqnarray}
where $C_1$ is a constant depending on $t_1$ and $C$ is a constant not depending on time. From (\ref{ibep1}) with $p=3$, we have, for all $ t\geq 0$,
\begin{equation}
\label{iie3.2}
\int_{M}e^{3u(t)}dV\leq C.
\end{equation}
 Using H\"{o}lder's inequality, (\ref{iie3.1}) and (\ref{iie3.2}), we obtain, for all $t_{0}\leq t\leq t_{1}$,

\begin{eqnarray*}
\int_{M}\left|\Delta u(t) -\frac{\rho}{\left|M\right|}\right|^{\frac{3}{2}}dV & \leq &\left(\int_{M}e^{-u(t)}\left(\Delta u (t)-\frac{\rho}{\left| M\right|}\right)^{2}dV  \right)^{\frac{3}{4}}\left(\int_{M}e^{3u(t)}dV \right)^{\frac{1}{4}} \\
& \leq &C_1, 
\end{eqnarray*}

\noindent therefore
\begin{equation}
\label{5dec2012e1}
 \int_{M}\left|\Delta u(t)\right|^{\frac{3}{2}}dV\leq C_1,\ \forall t_{0}\leq t\leq t_{1}.
 \end{equation}
We deduce from the Sobolev's embedding Theorem that 
\begin{equation}
\label{4juillet2012e1}
\left|u(t)\right|\leq C_1,\ \forall t_{0}\leq t\leq t_{1}.
\end{equation}
Derivating \eqref{isoE:flot} with respect to $t$, we see that $U(t)=\dfrac{\partial u(t)}{\partial t}$ satisfies
\begin{eqnarray}
\label{6juin2012e3}
\dfrac{\partial U(t)}{\partial t}&=&e^{-u(t)}\Delta U(t)-U(t)e^{-u(t)}\Delta u(t)\nonumber \\
&+&\frac{\rho}{\left|M\right|}U(t)e^{-u(t)}-\frac{\rho f}{\left(\int_{M}fe^{u(t)}dV\right)^{2}}\int_{M}U(t)fe^{u(t)}dV.
\end{eqnarray}
Now, using \eqref{6juin2012e3}, we have
\begin{eqnarray*}
\dfrac{\partial y(t)}{\partial t} & =& \dfrac{\partial }{\partial t}\left(\int_{M}U^{2}(t)e^{u(t)}dV\right) \\
& =& 2\int_{M}U(t)e^{u(t)}\left(e^{-u(t)}\Delta U(t)- U(t) e^{-u(t)} \Delta u(t) +\frac{\rho}{\left|M\right|}U(t)e^{-u(t)} \right. \nonumber\\
&-& \left.\frac{\rho f}{\left(\int_{M}fe^{u(t)}dV\right)^{2}}\int_{M}U(t)fe^{u(t)}dV \right) dV +\int_{M}U^{3}(t)e^{u(t)}dV. \nonumber
\end{eqnarray*}
Integrating by parts on $M$ and since $\Delta u(t)-\dfrac{\rho}{\left| M\right|}=U(t) e^{u(t)}-\dfrac{\rho f e^{u(t)}}{\int_{M}fe^{u(t)}dV}$, we obtain
\begin{eqnarray}
\label{lasteq}
\dfrac{\partial y(t)}{\partial t}
& =&-2 \int_{M}\left|\nabla U(t)\right|^{2}dV-2\int_{M}U^{2}(t)\left(\Delta u(t)-\frac{\rho}{\left| M\right|}\right) dV\nonumber\\
&-&\frac{2\rho}{\left(\int_{M}fe^{u(t)}dV\right)^{2}}\left(\int_{M}fU(t)e^{u(t)}dV\right)^{2} +\int_{M}U^{3}(t)e^{u(t)}dV \nonumber\\
& =&-2 \int_{M}\left|\nabla U(t)\right|^{2}dV- \int_{M}U^{3}(t)e^{u(t)}dV\nonumber\\
&+& 2\rho \left(\frac{\int_{M}fU^{2}(t)e^{u(t)}dV}{\int_{M}fe^{u(t)}dV}-\left( \frac{\int_{M}fU(t)e^{u(t)}dV}{\int_{M}fe^{u(t)}dV}\right)^{2} \right) .
\end{eqnarray}
\noindent Using the Gagliardo-Nirenberg's inequality, we get
\begin{equation}
\label{5dec2012e3}
\left\|U(t)\right\|_{L_{g_1(t)}^{3}\left(M\right)}\leq C \left\|U(t)\right\|_{L_{g_1(t)}^{2}\left(M\right)}^{\frac{2}{3}}\left\|U(t)\right\|_{H_{g_1(t)}^{1}\left(M\right)}^{\frac{1}{3}},
\end{equation}
where the norm are taken with respect to the conformal metric $g_{1}(t)=e^{u(t)}g$. Let's $\tilde{\lambda}_1(t)$ be the first eigenvalue of the laplacian with respect to the metric $g_1(t)$. By the Rayleigh quotient, we have
$$\tilde{\lambda}_1 (t)=\inf_{v\in H_{g_1(t)}^1 (M)}\dfrac{\int_M |\nabla^{g_1(t)}v|_{g_1(t)}^2dV_{g_1(t)}}{\int_M v^2 dV_{g_1(t)}}.$$
From \eqref{4juillet2012e1} and since $\displaystyle\int_M |\nabla^{g_1(t)}v|_{g_1(t)}^2dV_{g_1(t)}=\int_M |\nabla v|^2 dV$, we deduce that, $\forall t_0\leq t\leq t_1$,
\begin{equation}
\label{raleygh}
\tilde{\lambda}_1(t) \geq C_1.
\end{equation} 
From Poincaré's inequality, \eqref{raleygh} and since $\displaystyle\int_{M}U(t)e^{u(t)}dV=0$, we find, $\forall t_0\leq t\leq t_1$,
$$\int_M U^2(t) e^{u(t)}dV \leq \dfrac{1}{\tilde{\lambda}_1(t)}\int_M \left|\nabla U(t)\right|^{2}dV\leq C_1 \int_M \left|\nabla U(t)\right|^{2}dV,$$
hence, we get
\begin{equation}
\label{5dec2012e2}
\left\|U(t)\right\|_{H_{g_1(t)}^{1}\left(M\right)}\leq C_1 \left(\int_{M}\left|\nabla U(t)\right|^{2}dV\right)^{\frac{1}{2}}.
\end{equation}
Inserting \eqref{5dec2012e2} into \eqref{5dec2012e3}, we obtain 
$$\int_{M}e^{u(t)}\left|U(t)\right|^{3}dV\leq C_1 \left(\int_{M}U^{2}(t)e^{u(t)}dV\right)\left(\int_{M}\left|\nabla U(t)\right|^{2}dV\right)^{\frac{1}{2}}.$$
Putting the previous estimate into \eqref{lasteq} and using Young and Hölder's inequality, we find
\begin{eqnarray*}
\dfrac{\partial y(t)}{\partial t}&\leq & -2 \int_{M}\left|\nabla U(t)\right|^{2}dV +C_1 \left(\int_{M}U^{2}(t)e^{u(t)}dV\right)\left(\int_{M}\left|\nabla U(t)\right|^{2}dV\right)^{\frac{1}{2}}\\
&+& C \int_{M}U^{2}(t)e^{u(t)}dV+C\left(\int_{M}\left|U(t)\right|e^{u(t)}dV\right)^{2}\\
&\leq &  C_1 \left(\int_{M}U^{2}(t)e^{u(t)}dV\right)^{2}+C\left(\int_{M}U^{2}(t)e^{u(t)}dV\right).
\end{eqnarray*}
This implies, since $y(t_0)\leq \varepsilon$ and $y(t_1)=3\varepsilon$, that
$$2\varepsilon \leq y(t_{1})-y(t_{0})\leq C_1 \int_{t_{0}}^{t_{1}}y(t)dt.$$
Choosing $t_{0}$ sufficiently large, i.e. such that%, on a
$$C_1\int_{t_{0}}^{+\infty}y(t)dt\leq \varepsilon  ,$$
we obtain a contradiction. Thus, we have established that
\begin{equation}
\label{5dec2012e5}
y(t)=\int_M \left(\dfrac{\partial u}{\partial t}(t)\right)^2 e^{u(t)}dV\underset{t\rightarrow +\infty}{\longrightarrow} 0.
\end{equation}
Moreover, we get that all estimates we found during the proof, hold for all $t\geq 0$. In particular, we have that
$$\left|u(t)\right|\leq C,\ \forall t\geq 0.$$
Thus, from the previous estimate, \eqref{5dec2012e5} and using Young's inequality, we deduce that
\begin{eqnarray*}
\int_{M}\left(\Delta u(t) -\frac{\rho}{\left| M\right|}\right)^{2}dV&=&\int_M \left(\dfrac{\partial e^{u(t)}}{\partial t}-\dfrac{f e^{u(t_n)}}{\int_{M}fe^{u(t_n)}dV}\right)^2\\
&\leq & C y(t)+C\leq C,\ \forall t\geq 0.
\end{eqnarray*}
This implies that
$$\left\|u(t)\right\|_{H^{2}(M)}\leq C,\ \forall t\geq 0.$$
Therefore, there exist  $u_{\infty}\in H^{2}(M) $ and a sequence $\left(t_{n}\right)_n$, $t_n \underset{n\rightarrow +\infty}{\longrightarrow} +\infty$ such that
$$u(t_{n})\underset{n\rightarrow +\infty}{\longrightarrow} u_{\infty}\ weakly\ in\ H^{2}(M),$$
and
\begin{equation}
\label{5dec2012e6}
u(t_{n})\underset{n\rightarrow +\infty}{\longrightarrow} u_{\infty}\ \ in\ C^{\alpha}(M),\ \alpha\in (0,1) .
\end{equation}
It is easy to check that $u_{\infty}$ is a weak solution of
\begin{equation*}
\label{eqfin1}
\Delta u_{\infty} -\frac{\rho}{\left|M\right|}+\frac{\rho f e^{u_{\infty}}}{\int_{M}fe^{u_{\infty}}dV}=0,
\end{equation*}
and, by boothstrap regularity arguments, we have $u_{\infty}\in C^\infty (M)$. We claim that $\left\|u(t_{n})- u_{\infty}\right\|_{H^2(M)}\underset{n\rightarrow +\infty}{\longrightarrow} 0$ strongly in $H^{2}(M)$. To prove the claim, in view of \eqref{5dec2012e6}, it is sufficient to show that $\displaystyle\int_M \left(\Delta u(t_n) -\Delta u_\infty \right)^2 dV\underset{n\rightarrow +\infty}{\longrightarrow} 0$. Using Hölder's inequality, we have 
\begin{eqnarray}
\label{5dec2012e10}
\int_M \left(\Delta u(t_n) -\Delta u_\infty \right)^2 dV
&=&\int_M \left(\rho \left(\dfrac{f e^{u_{\infty}}}{\int_{M}fe^{u_{\infty}}dV}-\dfrac{f e^{u(t_n)}}{\int_{M}fe^{u(t_n)}dV} \right)+ \dfrac{\partial e^{u(t_n)}}{\partial t} \right)^2 dV\nonumber\\
&\leq & C \int_M \left(\dfrac{ e^{u_{\infty}}}{\int_{M}fe^{u_{\infty}}dV}-\dfrac{ e^{u(t_n)}}{\int_{M}fe^{u(t_n)}dV} \right)^2dV\nonumber \\
&+& C \int_M \left(\dfrac{\partial e^{u(t_n)}}{\partial t} \right)^2 dV.
\end{eqnarray}
Since $|u(t)|\leq C$, $\forall t\geq 0$, we deduce from \eqref{5dec2012e5} that
\begin{equation}
\label{5dec2012e11}
\int_M \left(\dfrac{\partial e^{u(t_n)}}{\partial t} \right)^2 dV\underset{n\rightarrow +\infty}{\longrightarrow}0. 
\end{equation}
Let's denote $\beta = \dfrac{\int_{M}fe^{u(t_n)}dV}{\int_{M}fe^{u_\infty}dV}$. Using the estimate $\left|e^{x}-1\right|\leq \left|x\right| e^{x}$ and Hölder's inequality, we have
\begin{eqnarray*}
&&\int_{M}\left(\dfrac{ e^{u_{\infty}}}{\int_{M}fe^{u_{\infty}}dV}-\dfrac{ e^{u(t_n)}}{\int_{M}fe^{u(t_n)}dV} \right)^2dV\\
& =&\dfrac{1}{\left(\int_{M}fe^{u(t_n)}dV\right)^2}\int_{M}e^{2u\left(t_{n}\right)}\left|e^{u_{\infty}-u(t_n)+\ln \beta}-1\right|^2 dV \\
& \leq & C\int_{M}e^{2u\left(t_{n}\right)}  \left|u_\infty- u(t_n)+\ln \beta\right|^2 e^{2\left(u_{\infty}-u(t_n)+\ln \beta\right)} dV \\
& \leq &C \left(\int_{M}e^{8u\left(t_{n}\right)}dV\right)^{\frac{1}{4}}\left( \int_{M} \left|u_\infty- u(t_n)+\ln \beta\right|^{2} dV\right)^{\frac{1}{2}}\\
&\times &\left(\int_{M}e^{8(u_\infty- u(t_n)+\ln \beta)} dV\right)^{\frac{1}{4}}.
\end{eqnarray*}
From the previous inequality, \eqref{5dec2012e6} and since $\displaystyle\int_{M}fe^{u(t_n)}dV\underset{n\rightarrow +\infty}{\longrightarrow}\int_{M}fe^{u_{\infty}}dV$, we obtain that
\begin{equation}
\label{5dec2012e12}
\int_{M}\left(\dfrac{ e^{u_{\infty}}}{\int_{M}fe^{u_{\infty}}dV}-\dfrac{ e^{u(t_n)}}{\int_{M}fe^{u(t_n)}dV} \right)^2dV\underset{n\rightarrow +\infty}{\longrightarrow} 0.
\end{equation}
Therefore, from \eqref{5dec2012e10}, \eqref{5dec2012e11} and \eqref{5dec2012e12}, we deduce that
$$\int_M \left(\Delta u(t_n) -\Delta u_\infty \right)^2 dV\underset{n\rightarrow +\infty}{\longrightarrow} 0.$$
Thus, we proved that $\left\|u(t_{n})- u_{\infty}\right\|_{H^2(M)}\underset{n\rightarrow +\infty}{\longrightarrow} 0$ strongly in $H^{2}(M)$.
Finally, since the flow \eqref{isoE:flot} is a gradient flow and $E_f$ is real analytic, by using a general result from Simon \cite{MR727703}, we derive that
$$\left\|u(t)-u_\infty \right\|\underset{t\rightarrow +\infty}{\longrightarrow}0.$$
This concludes the proof of Theorem \ref{isoconvergence}.
\bibliographystyle{plain}
\bibliography{biblio}
\end{document}